%% file: PPW_Space_Groups_arxiv.tex
\documentclass[a4paper]{article}
\input{macros}
\input{environs}

\usepackage{xcolor}
\usepackage{inputenc}
\usepackage{enumerate}

\DeclareMathOperator{\Isom}{Isom}
\DeclareMathOperator{\Aff}{Aff}

\newcommand{\set}[1]{\{#1\}}
\newcommand{\cgp}{crystallographic group}
\newcommand{\cgps}{crystallographic groups}

\title{Distinguishing crystallographic groups by their finite quotients}
\author{Pawe\l\ Piwek, David Popovi\'c and Gareth Wilkes} 

\begin{document}
\maketitle
\begin{abstract}
Using the computer algebra program \texttt{GAP}, we show that all crystallographic groups in dimensions at most 4 are distinguished from each other by their sets of finite quotients.
\end{abstract}	
	\section{Introduction}
	The classification of the isomorphism types of \cgps \ in a given dimension is an old problem. Fedorov and Schonflies gave in 1895 a list of the 219 \cgp \ types in three dimensions. This list was later extended to dimension 4 \cite{BBNWZ}, and with the assistance of computers Plesken and Schultz \cite{PS} have calculated the number of space groups in dimensions 5 and 6 and presented an algorithm for constructing them on demand.
	
    There has been much recent study of whether residually finite groups, or classes of residually finite groups, may be distinguished from each other by their sets of finite quotient groups---particularly those groups connected with low-dimensional geometry or topology. Examples of pairs of distinct groups with the same sets of finite quotients have been found by Baumslag \cite{baum74}, Hempel \cite{hempel14}, Stebe \cite{stebe72} and others. `Rigidity' results showing that certain groups are distinguished by finite quotients (either among all residually finite groups or among some subclass) include \cite{BCR14, BMRS, Wilkes17a}.
	
   The examples of Baumslag \cite{baum74} just cited are virtually abelian (that is, possess an abelian normal subgroup of finite index). The crystallographic groups are a natural interesting family of virtually abelian groups, so we ask whether such groups can share the same families of finite quotients. (The examples of Baumslag, of the form $\Z/25\Z \rtimes \Z$, are not crystallographic groups.)	
	
	\begin{question}\label{mainquestion}
		For non-isomorphic \cgps \ $\Gamma_1$ and $\Gamma_2$, must $\Gamma_1$ and $\Gamma_2$ have different sets of finite quotient groups?
	\end{question}
	
This question is known to have a negative answer in the generality stated above: see \cite{FNP80} for an example of non-isomorphic crystallographic groups of dimension 22 with the same finite quotients. 

	This paper studies Question \ref{mainquestion} for \cgps\ in dimensions $n=2$, $3$ and $4$ and gives a positive answer. We perform computations with the assistance of \texttt{GAP}, using methods closely related to the classical Fundamental Theorem of Crystallography.

It should be mentioned that the question of whether crystallographic groups can share the same set of finite questions (or, in other terminology, `belong to the same genus') has been considered elsewhere. In particular see Holt and Plesken \cite{HP89} for an extensive discussion of the genus of {\em perfect} space groups in dimensions at most 10.
	
	\section{Background on \cgps}
	
	In this section we introduce the basic results on \cgps, using \cite{Hiller} and \cite{Szcz} as sources. Let $d\geq 1$ be an integer.
	
	\begin{defn}
		Let $\Isom(d)$ denote the group of isometries of $\R^d$ (with the Euclidean metric). A subgroup $\Gamma \le \Isom(d)$ which is discrete and cocompact is called a \emph{\cgp\ of dimension $d$}. A crystallographic group is sometimes also called a {\em space group}.
	\end{defn}
	
	A choice of basepoint $0$ for $\R^d$ gives an identification of $\Isom(d)$ as a semidirect product $\R^d \rtimes \mathrm{O}(d)$, with the following notation for group elements. The isometry $x \mapsto a+ Ax$ for $a\in \R^d$ and $A \in \mathrm{O}(d)$ will be denoted by $(a,A)$. The semidirect product structure is given by:
	$$(a,A) (b,B) = (a+Ab,AB).$$
	The identification $\Isom(d) \iso \R^d\rtimes \mathrm{O}(d)$ is not unique, but the normal subgroup $\R^d \nsgp \Isom(d)$ of translations is the same for all such choices of identification. In consequence, the following is a meaningful definition.
	
	\begin{defn}
		Let $\Gamma$ be a \cgp{} of dimension $d$. The group of \emph{translations in $\Gamma$} is the subgroup $M = \set{(a,I) \in \Gamma}$.
	\end{defn}
	
	\begin{theorem}[First Bieberbach Theorem. See Section 2.1 of \cite{Szcz}.]
		\label{thm:1stBieberbach}
		Let $\Gamma$ be a \cgp{} of dimension $d$ with group of translations $M \le \Gamma$. Then:
		\begin{enumerate}[(i)]
			\item $M$ is isomorphic to $\Z^d$;
			\item $M$ is maximal abelian; and
			\item $M$ is normal in $\Gamma$ with finite index.
		\end{enumerate}
	\end{theorem}
	
It is important to note that these properties {\em characterize} the translation subgroup.

\begin{prop}
Let $G$ be a group and let $A$ and $B$ be torsion-free abelian subgroups of $G$ which are finite index, normal, and maximal abelian. Then $A= B$.
\end{prop}	
\begin{proof}
Suppose there is an element $b\in B\smallsetminus A$. We show that $\gp{b, A}$ is abelian, a contradiction. Let $a\in A$. There exists $n\in\N$ such that $a^n$ lies in the finite index subgroup $B$ of $G$. Then $a^n = b^{-1}a^n b = (b^{-1}ab)^n$. By normality, $b^{-1}ab\in A$, so since $A$ is torsion-free abelian it follows that $a= b^{-1}ab$ as required.
\end{proof}
\begin{clly}\label{prop:uniqueMAsubgp}
The translation subgroup of a crystallographic group $\Gamma$ is the {\em unique} finite index normal torsion-free maximal abelian subgroup of $\Gamma$.
\end{clly}
	
	\begin{defn}
		Let $\Gamma$ be a \cgp. We define its \emph{point group} to be:
		$$G = \set{A \in \mathrm{O}(d)\mid(a,A) \in \Gamma \:\textrm{for some} \: a \in \R^d}.$$
	\end{defn}
	
	\begin{clly}
		Let $\Gamma$ be a \cgp{} with point group $G$ and group of translations $M$. Then $G$ is isomorphic to the finite group $\Gamma / M$.
	\end{clly}
	
	In order to check properties of crystallographic groups algorithmically, it is of course important that there are only finitely many of any given dimension.
	
	\begin{theorem}[Second Bieberbach Theorem, Section 2.2 of \cite{Szcz}]
		\label{thm:2ndBieberbach}
		For any number $d$, there are only finitely many isomorphism classes of \cgps{} of dimension $d$.
	\end{theorem}
		
Let $\Gamma$ be a crystallographic group. Due to Theorem~\ref{thm:1stBieberbach}, there is a short exact sequence
	\begin{equation}
	\label{eq:SES}
	0 \to M \to \Gamma \to G \to 1
	\end{equation}
	where $M$ and $G$ are the translation and point groups of $\Gamma$ respectively. By convention we write $M$ additively and $G$ multiplicatively. Since $M$ is a normal subgroup, $\Gamma$ acts on it by conjugation:
	$$(a,A)(b,I)(a,A)^{-1} = (a+Ab,A)(-A^{-1}a,A^{-1}) = (Ab, I).$$
	
	It is clear that the translation part $a$ acts trivially, so in fact, we can consider this to be an action of $G$ on $M$. The action is faithful since, again according to Theorem~\ref{thm:1stBieberbach}, $M$ is a maximal abelian subgroup. This allows us to treat $M$ as a $\Z G$-module and, after choosing a basis for $M\iso\Z^d$, identify $G$ with a subgroup of $\GL{d}{\Z}$.
	
	\begin{defn}
		For a given $G \le \mathrm{O}(d)$ that acts on a $d$-dimensional lattice $M \le \R^d$ we define the \emph{(G,M) crystal class} to be the set of all space groups $\Gamma$ with $M$ being the group of translations and $G$ being the point group. 
		
		We say that two crystal classes $(G,M)$ and $(G',M')$ are \emph{\Z-equivalent} if there exist isomorphisms $\psi\colon G\to G'$ and $\phi\colon M \to M'$ such that 
		\[\phi(g\cdot m) = \psi(g)\cdot \phi(m) \]
for all $g\in G, m\in M$. When we regard $G$ and $G'$ as subgroups of $\GL{d}{\Z}$, this is equivalent to saying that $G$ and $G'$ are conjugate within $\GL{d}{\Z}$.
		
		We define the relation of \emph{\Q-equivalence}, as conjugacy of $G$ and $G'$ in \GL{d}{\Q}.
		
We abbreviate `\Z-(\Q-)equivalence class' to `\Z-(\Q-)class'.	Alternative terms for \Z-class and \Q-class are {\em arithmetic crystal class} and {\em geometric crystal class} respectively.
	\end{defn}
		
	\begin{thmquote}[Jordan--Zassenhaus: see \cite{CR66}, Theorem 79.1]
		For any integer $d\geq 1$, there are finitely many conjugacy classes of finite subgroups of $\GL{d}{\Z}$. In particular, there are only finitely many possible $\Z$-equivalence classes of a given dimension.
	\end{thmquote}
	We must also be able to quantify how many \cgps{} up to isomorphism there are in each arithmetic crystal class.
		
	Each space group is an extension of $G$ by $M$, and such extensions are classified (up to equivalence of extensions) by elements of $H^2(G,M)$ \cite[Section IV.3]{Brown}, which is a finite group \cite[Proposition III.10.1]{Brown}. It remains to decide which extension classes give rise to isomorphic groups.
	
	\begin{theorem}[Third Bieberbach Theorem, Section 2.3 of \cite{Szcz}]
		\label{thm:3rdBieberbach}
		Any two \cgps{} $\Gamma, \Gamma' \le \Isom(d)$ that are isomorphic as abstract groups are also conjugate in the group $\Aff(d)$ of affine homeomorphisms of $\R^d$.
	\end{theorem}
	
	Theorem~\ref{thm:3rdBieberbach} tells us that we may classify all \cgps{} $\Gamma$ in the $(G,M)$ crystal class by considering the action on $H^2(G,M)$ of affine maps that preserve $M$ and $G$. This boils down to the following statement. See for example \cite[Theorem 5.2]{Hiller} for an elementary account.
	
	\begin{theorem}[Fundamental Theorem of Mathematical Crystallography]
		Let $(G,M)$ be a crystal class. There is a one-to-one correspondence between the isomorphism classes of groups in the crystal class $(G,M)$ and the orbits of the action of $N = {\cal N}_{\Aut(M)}(G)$ on $H^2(G,M)$.
	\end{theorem}

\begin{rmk}
Here and elsewhere in the paper the notation ${\cal N}_G(H)$ denotes the normalizer of a subgroup $H$ of a group $G$.
\end{rmk}

	\section{Basic results on profinite completions}
	
	This section introduces elementary results on profinite groups. The main reference is the book \cite{RZ}.
	
	\begin{defn}
		An \emph{inverse set} is a partially ordered set $I$ such that for every $i, j \in I$ there exists $k \in I$ such	that $k \leq i$ and $k \leq j$. An \emph{inverse system} is a family of sets $\{X_i\}_{i \in I }$, where $I$ is a directed set, and a family of maps $\phi_{ij} : X_i \to X_j$ whenever $i \leq j$, such that:
		\begin{itemize}
			\setlength\itemsep{0pt} 
			\item  $\phi_{ii} = \id_{X_i}$
			\item  $\phi_{jk} \phi_{ij} = \phi_{ik}$ whenever $i \leq j \leq k$.
		\end{itemize}
		Denoting this system by $(X_i, \phi_{ij}, I)$, the \emph{inverse limit} of the inverse system $(X_i, \phi_{ij}, I)$ is the set 
		$$ \invlim X_i = \{ (x_i) \in	\prod_{i \in I}^{} X_i \ \mid \ \phi_{ij}(x_i)=x_j \text{ whenever } i \leq j \}. $$
	\end{defn}
	
When the $X_i$ are groups and the $\phi_{ij}$ are group homomorphisms, the resulting inverse limit is a group. A profinite completion of a group is the following case of this general construction.
	
	\begin{defn}
		Let $G$ be a finitely generated group and let ${\cal N}$ denote the collection of all finite index normal subgroups of $G$. We can make ${\cal N}$ into an inverse set by declaring that for $N_i, N_j \in {\cal N}$ we have $N_i \leq N_j$ whenever $N_i \subseteq N_j$. In this case, there are natural epimorphisms $\phi_{N_i N_j} : G/N_i \to G/N_j$. The inverse limit of the inverse system $(G/N_i, \phi_{N_i N_j}, {\cal N})$ is the \emph{profinite completion} of $G$ and is denoted by $\widehat{G}$. 
		\end{defn}	
		Note that the maps $G\to G/N_i$ give a natural map $G\to \widehat G$. This map is injective if and only if $G$ is {\em residually finite}.
\begin{defn}
A group $G$ is {\em residually finite} if, for every $g\in G\smallsetminus \{1\}$, there exists some finite index normal subgroup $N\nsgp G$ such that $g\notin N$.
\end{defn}
	
	We equip the quotients $G/N_i$ with discrete topology and consider the topology on $\widehat{G}$ that is induced by the product topology on $\prod_{N_i \in N} G/N_i$. The following are standard facts about the topology of profinite completions.
	
	\begin{prop}[\cite{RZ}, Theorem 2.1.3]
		The profinite completion $\widehat{G}$ is a compact, Hausdorff, totally disconnected topological group.
	\end{prop}
	
	\begin{prop}
		Let $\widehat{G}$ be a profinite completion of a group $G$.
		\begin{enumerate}[(i)]
			\setlength\itemsep{0 pt}
			\item \cite[Lemma 2.1.2]{RZ} A subgroup $U \leq \widehat{G}$ is open if and only if it is closed of finite index.
			\item \cite[Proposition 2.1.4(d)]{RZ} A subgroup $H \leq \widehat{G}$ is closed if and only if it is the intersection of all open subgroups of $\widehat{G}$ containing $H$.
		\end{enumerate}
		
	\end{prop}

	Profinite completions contain all information about finite quotients of a group in the following sense.
	\begin{prop}[\cite{RZ}, Corollary 3.2.8]
		Let $G_1$ and $G_2$ be finitely generated abstract groups. Then $\widehat{G}_1$ and $\widehat{G}_2$ are	isomorphic as topological groups if and only if the sets of (isomorphism types of) finite quotients of $G_1$ and $G_2$ coincide.
	\end{prop}
	
\begin{prop}[Proposition 3.2.2 of \cite{RZ}]\label{BasicCorr}
		If $G$ is a finitely generated residually finite group, then there is a one-to-one
		correspondence between the set $\mathcal{X}$ of subgroups of $G$ with finite index, and the set $\mathcal{Y}$ of all open subgroups of $\widehat{G}$. Identifying $G$ with its image in the completion, this correspondence is given by:
		\begin{itemize}
			\setlength{\itemsep}{0 pt}
			\item For $H \in \mathcal{X}$, $H \mapsto \overline{H}$.
			\item For $Y \in \mathcal{Y}$, $Y \mapsto Y \cap G$.
		\end{itemize}
		Moreover, 
		\begin{itemize}
			\item If $H,K \in \mathcal{X}$ and $K \leq H$ then $[H : K] = [\overline{H} : \overline{K}]$. 
			\item For $H\in\cal X$, $H$ is normal if and only if $\overline H$ is normal.
			\item For $H\in\cal X$, $\overline H$ is isomorphic to $\widehat H$.
		\end{itemize}
		
	\end{prop}
	
It is elementary that a finitely generated virtually abelian group is residually finite. 	In particular, crystallographic groups are residually finite. The profinite completion of a residually finite group is abelian if and only if the group itself is abelian. Hence Proposition \ref{BasicCorr} and Corollary \ref{prop:uniqueMAsubgp} together imply the following proposition.

\begin{prop}\label{UniqueOpenMASubgp}
Let $\Gamma$ be a crystallographic group with translation group $M$. The only open normal torsion-free maximal abelian subgroup of $\widehat\Gamma$ is $\widehat M$.
\end{prop}
	
	\section{Profinite properties of crystallographic groups}

Let $\Gamma_1$ and $\Gamma_2$ be crystallographic groups of dimension $d$, with translation groups $M_1$ and $M_2$ and point groups $G_1$ and $G_2$ respectively. Suppose $\widehat \Gamma_1\iso \widehat\Gamma_2$. By Proposition \ref{UniqueOpenMASubgp} there must be a short exact sequence of isomorphisms
		\begin{equation}\label{eqnIsoOfSES}
		 \begin{tikzcd}
	0 \ar{r} & \widehat M_1 \ar{r} \ar{d}{\phi}[swap]{\iso}& \widehat \Gamma_1 \ar{r}\ar{d}{\Phi}[swap]{\iso} & G_1 \ar{r} \ar{d}{\psi}[swap]{\iso}& 0 \\
	0 \ar{r} & \widehat M_2 \ar{r} &\widehat \Gamma_2 \ar{r} & G_2 \ar{r} & 0 
	\end{tikzcd}
\end{equation}		
The $G_1$- and $G_2$-modules $\widehat M_1$ and $\widehat M_2$ are thus isomorphic modules in the sense that
\[\phi(g\cdot m) = \psi(g)\cdot \phi(m) \]
for all $g\in G_1, m\in \widehat M_1$.

 Given a $\Z$-basis for each $M_i$ (and hence a $\widehat\Z$-basis for $\widehat M_i$), so that the $G_i$ are realised as subgroups of \GL{d}{\Z}, this isomorphism states that $G_1$ and $G_2$ are conjugate as subgroups of the group \GL{d}{\widehat\Z} of continuous automorphisms of $\widehat \Z{}^d$.	
	
Hence a first key question to be answered	when deciding whether crystallographic groups are distinguished by their profinite completions is the following.
\begin{question}
Let $G_1$ and $G_2$ be finite subgroups of \GL{d}{\Z} which arise as crystallographic point groups. If $G_1$ is conjugate to $G_2$ in \GL{d}{\widehat\Z}, must $G_1$ be conjugate to $G_2$ in \GL{d}{\Z}?
\end{question}

As with Question \ref{mainquestion}, the general answer to this question is `no': the same citation \cite{FNP80} provides an example in dimension 22. We will continue to focus on our computational approach to low dimensions.
	
For the following discussion suppose that the answer to this question is affirmative in dimensions at most 4---so that \cgps\ with the same profinite completion are necessarily in the same \Z-class. This assumption is addressed in Section \ref{sec:ConjZHat} below. We classify the profinite isomorphism types of space groups in a given arithmetic class in terms of an action on cohomology; this analysis is almost identical to the discrete case, but we nevertheless include it for completeness.
	
	Let $\Gamma_1$ and $\Gamma_2$ be space groups in the same arithmetic class, so that both may be considered to be extensions of a given finite group $G$ by a given $G$-module $M\iso \Z^d$. Suppose there is a continuous isomorphism $\Phi\colon \widehat \Gamma_1 \to \widehat \Gamma_2$. By Proposition \ref{UniqueOpenMASubgp}, the subgroup $\overline M = \widehat M$ of $\widehat \Gamma_i$ is the unique maximal abelian open normal subgroup of $\widehat \Gamma_i$, so that $\Phi$ induces an isomorphism of short exact sequences as in \eqref{eqnIsoOfSES}, with $\widehat{M}=\widehat{M}_1=\widehat{M}_2$.
	
	To describe these groups as cohomology classes in $H^2(G, \widehat M)$ represented by particular cocycles $G^2\to \widehat M$, choose a section $\sigma_1 \colon G \to \widehat \Gamma_1$. Define also a section $\sigma_2 \colon G \to \widehat \Gamma_2$ by 
	\[ \sigma_2 = \Phi\circ \sigma_1 \circ \psi^{-1} \]
	The action of $G$ on $\widehat M$ is induced by conjugation in the group $\widehat \Gamma_i$ via the formula
	\[g\cdot m = \sigma_i(g) m \sigma_i(g)^{-1} \] 
	By Theorem~\ref{thm:1stBieberbach}, this action is faithful, so the natural map $G\to \Aut(M) \subseteq \Aut(\widehat M)$ is injective, and we may consider $G$ to be a subgroup of $\Aut \widehat M$. We note that $\phi$ is an element of ${\cal N}_{\Aut \widehat M}(G)$, conjugation by which induces the map $\psi$, for:
	\begin{eqnarray*}
		\psi(g) \cdot m & = & (\sigma_2\psi(g)) m (\sigma_2\psi(g))^{-1} \\
		& = & \left( \Phi\sigma_1\psi^{-1} \psi (g) \right) m \left( \Phi\sigma_1\psi^{-1} \psi (g)\right)^{-1}\\
		& = & \Phi\big( \sigma_1(g) \Phi^{-1} m \sigma_1(g)^{-1} \big) \\
		& = & \phi ( g\cdot \phi^{-1}(m))  
	\end{eqnarray*}
	whence $\psi(g) = \phi g \phi^{-1}$ in $\Aut\widehat M$. 
	
	At the level of cocycles, $\widehat{\Gamma}_i$ is represented by a 2-cocycle $\zeta_i$ given by
	\[ \zeta_i (g_1, g_2) = \sigma_i(g_1g_2)(\sigma_i(g_1)\sigma_i(g_2))^{-1}\]
	Using the definition of $\sigma_2$, we now find that 
	\begin{equation}\label{ActionOfNormOnCocycles}
	\zeta_2(g_1, g_2) = \phi(\zeta_1(\phi^{-1}g_1\phi, \phi^{-1}g_2\phi))
	\end{equation}
	
This formula defines the standard action of ${\cal N}_{\Aut\widehat M}(G)$ on $H^2(G,\widehat M)$, so $\zeta_1$ and $\zeta_2$ are in the same orbit of this action.

	Conversely, given two extensions $\widehat \Gamma_1$ and $\widehat \Gamma_2$ of $G$ by $\widehat M$, with sections $\sigma_1$ and $\sigma_2$ yielding representative cocycles $\zeta_1$ and $\zeta_2$, if we have some $\phi\in{\cal N}_{\Aut\widehat M}(G)$ such that the relation \eqref{ActionOfNormOnCocycles} holds, it may be readily verified that the map $\Phi\colon \widehat \Gamma_1 \to \widehat \Gamma_2$ defined by 
	\[\Phi(m\sigma_1(g)) = \phi(m)\cdot \sigma_2(\phi g \phi^{-1}) \]
	is a continuous isomorphism.
	
	This concludes the proof of the following statement.
	\begin{prop}
		Let $(G, M)$ be an arithmetic class of space groups. The isomorphism types of profinite completions of space groups in this class are in bijection with the orbits of $H^2(G, \widehat M)$ under the action of ${\cal N}_{\Aut\widehat M}(G)$.
	\end{prop}
	\section{Computational considerations}
	
	The previous section gave us a two-part formulation of our task to decide whether space groups in dimension $d$ have distinct profinite completions:
	\begin{itemize}
		\item to decide whether distinct conjugacy classes of finite subgroups of \GL{d}{\Z} may become conjugate in \GL{d}{\widehat \Z}; and
		\item to compute the action of a normalizer of a finite group $G$ in \GL{d}{\widehat\Z} on the cohomology group $H^2(G, \widehat\Z {}^d)$. 
	\end{itemize} 
	In this section we describe the methods for doing this in practice, and which allowed us to answer Question \ref{mainquestion} in the affirmative in dimensions 2, 3 and 4.
	
	\subsection{The \texttt{GAP} package \texttt{CrystCat}}
	We make use of the program \texttt{GAP} \cite{GAP4}, equipped with the package \texttt{CrystCat} \cite{CrystCat} to enable our computations. Among the data provided by this package, of which we make use, are:
	\begin{itemize}
		\item A list of \Q-classes of crystallographic point groups in dimensions $d=2$, 3 or 4;
		\item for each \Q-class, a list of the \Z-classes contained within it, and a representative finite subgroup of \GL{d}{\Z} for each \Z-class;
		\item finite presentations for the point group of a \Z-class, consistent across all \Z-classes in a given \Q-class;
		\item the number of space group types in a given \Z-class.
	\end{itemize}
	
\subsection{Conjugacy problems in $\widehat{\Z}$}\label{sec:ConjZHat}
	
	Of key importance in the present paper is the ability to decide conjugacy problems over the profinite integers $\widehat \Z$ or the $p$-adic integers \Z[p]. Let ${\cal A} = \{ A_1,\ldots, A_k\}$ and ${\cal B} = \{B_1, \ldots, B_k\}$ be tuples of $d\times d$ integral matrices. Let $\pi$ be some finite set of primes and let \Z[\pi] denote the ring \[ \Z[\pi] = \prod_{p\in \pi} \Z[p]. \]We wish to decide computationally whether there exists some $X \in \GL{d}{\Z[\pi]}$ such that $X^{-1}A_i X = B_i$ for all $i$. By \cite[Remark 4.1.7]{HP89} it in fact suffices to consider $\pi$ to be a subset of those primes dividing $|G|$; for other primes $p\nmid |G|$ one automatically has conjugacy over \Z[p]. This may profitably be rephrased as a linear problem: let $T_{\cal A, B}$ be the $kd^2 \times d^2$ integer matrix  corresponding to the map
	\[T_{\cal A, B}\colon \Z[\pi]^{d^2} \to (\Z[\pi]^{d^2}){}^k, \quad  X \mapsto (A_iX - XB_i)_{i \in I}\]
	where we implicitly identify the set of $m\times n$ matrices over a ring $R$ with $R^{mn}$. We now have to consider the linear condition $T_{\cal A, B}(X) = 0$, together with the non-linear condition that $\det(X) \in \Z[\pi]^{\!\times}$. We break this into two stages: a solution modulo $q = \prod\{p \in \pi\}$ and the question of whether this solution may be `lifted' to a solution over \Z[\pi]. 
	
	The business of seeking a solution modulo $q$ is a finite question: we simply solve the linear equation $A_iX - XB_i=0$ over each field $\F_p$, and discard those solutions which have determinant zero. If no solutions remain for some $p$, then we are done: the tuples $\cal A$ and $\cal B$ are not conjugate in $\GL{d}{\Z[\pi]}$. Otherwise one may use the Chinese Remainder Theorem to produce a list of the elements of $\GL{d}{\Z/q\Z}$ which  conjugate $\cal A$ to $\cal B$. We wish to discover whether any of these elements is the modulo $q$ reduction of an element of $\GL{d}{\Z[\pi]}$ which conjugates $\cal A$ to $\cal B$.
	
	Now suppose $\overline X_0$ is some matrix in $\GL{d}{\Z/q\Z}$ which conjugates $\cal A$ to $\cal B$ modulo $q$. We may take some integer matrix $X_0$ which maps to $\overline X_0$ on reduction modulo $q$. Then there exist integer matrices $E_i$ such that 
	\[ A_i X_0 - X_0 B_i = qE_i\]
	The elements of $\GL{d}{\Z[\pi]}$ which are equivalent to $X_0$ modulo $q$ are those of the form $X_0 + qY$ where $Y$ is any $d\times d$ matrix over \Z[\pi]. All these matrices have determinant coprime to $q$, so are indeed invertible over \Z[\pi]. Such a matrix conjugates $\cal A$ to $\cal B$ if and only if 
	\[0 =  A_i (X_0+qY) - (X_0+qY) B_i = q(E_i+ A_i Y - YB_i)\]
	for each $i$. We must therefore decide whether $(E_i)_{i\in I}$ is contained in the image of $T_{\cal A, B}$ over \Z[\pi]. 
	
	Let $M$ be an $n\times m$ integer matrix and $\pi$ a set of primes. We ask which integer points $x\in \Z^{m}$ are in the image of the map $T_M\colon \Z[\pi]^n\to \Z[\pi]^m$ specified by $M$ and solve the question in the following manner. Let $S$ be the Smith normal form of $M$ over $\Z$, and let $P$ and $Q$ be the invertible integer matrices such that $M = PSQ$. Since $P$ and $Q$ are invertible over $\Z$, it follows that $x\in \Z^{m}$ is in the image of  $T_M$ if and only if $P^{-1}(x)$ is in the image of the map $T_S \colon \Z[\pi]^n\to \Z[\pi]^m $ specified by $S$. The Smith normal form $S$ has no off-diagonal elements; let the non-zero on-diagonal entries be $d_1,\ldots, d_r$. Then the image of the map $T_S$ is plainly a sum 
	\[ d_1 \Z[\pi] \oplus \cdots \oplus d_r\Z[\pi] \]
	whose integral points are 
	\[ \big(d_1 \Z[\pi] \cap \Z\big)\oplus \cdots \oplus \big(d_r\Z[\pi]\cap \Z\big) \]
	Write each $d_i$ as a product $\bar d_i \cdot e_i$, where $\bar d_i$ is divisible only by primes from $\pi$ and $e_i$ is not divisible by any element of $\pi$. Then $d_i \Z[\pi] \cap \Z = \bar d_i \Z$. Form therefore a new matrix $\bar S$ by replacing each non-zero entry $d_i$ of $S$ with $\bar d_i$. The integral points of the image of $T_M$ are now exactly the image over $\Z$ of the integral matrix $P\bar S$, which may be computed without difficulty.
	
	Being thus in possession of a method for deciding conjugacy problems in \GL{d}{\Z[\pi]}, we make a few remarks on its implementation in the present case to show that particular finite subgroups of \GL{d}{\Z} are not conjugate in \GL{d}{\widehat \Z}. 
	
First observe that conjugacy in \GL{d}{\Q} is a necessary condition for conjugacy over $\widehat \Z$---or indeed over any \Z[p].	Suppose that the tuples of integer matrices $\cal A$ and $\cal B$ are conjugate in \GL{d}{\Q[p]}. Consider the kernel of the map 
	\[T_{\cal A, B}\colon \Q[p]^{d^2} \to (\Q[p]^{d^2}){}^k, \quad  X \mapsto (A_iX - XB_i)_{i \in I}\]
This map is specified by an integer matrix (with respect to the standard bases of the given vector spaces). It follows that the intersection $\Q^{d^2}\cap \ker T_{\cal A, B}$ is dense in the whole kernel.  This kernel has, by assumption, non-trivial intesection with the set of matrices with non-zero determinant---which is an open subset of $\Q[p]^{d^2}$. Therefore there is some element of $\Q^{d^2}\cap \ker T_{\cal A, B}$ with non-zero determinant. This element exhibits that $\cal A$ and $\cal B$ are conjugate in \GL{d}{\Q}.

	Let $G_1$ and $G_2$ be two finite subgroups of \GL{d}{\Z}, which are conjugate in \GL{d}{\Q}. Let us suppose further that $G_1$ and $G_2$ are equipped with generating sets $X_1$ and $X_2$, such that some isomorphism of $G_1$ with $G_2$ carries $X_1$ to $X_2$. We remark that this is precisely the situation arising in our analysis of point groups given by the catalogue \texttt{CrystCat}: two point groups corresponding to two \Z-classes contained in the same \Q-class come equipped with generating sets of the type described above.
	
	Any isomorphism $G_1 \to G_2$ must then take $X_1$ to $\sigma(X_2)$ for some  $\sigma \in \Aut(G_2)$. Hence to decide whether $G_1$ and $G_2$ are conjugate in \GL{d}{\widehat\Z} one should check for each $\sigma$ whether the tuples $X_1$ and $\sigma(X_2)$ are conjugate in each \GL{d}{\Z[p]}. We do this for some suitable finite set of primes $p\in\pi$ by applying the above machinations to the tuples ${\cal A} = X_1$  and ${\cal B} = \sigma(X_2)$ for each $\sigma$.

In our code, to minimize the computational burden, we adopted the following sequential approach to attempt to show that the point groups $G_1$ and $G_2$ (corresponding to \Z-classes in the same \Q-class) are not conjugate over $\widehat\Z$:
	\begin{enumerate}
		\item First compute the cohomology groups \[H^1(G_i, \Z^d/ |G_i|\Z^d)\quad \text{and}\quad H^2(G_i, \Z^d/ |G_i|\Z^d)\] (where the coefficient groups are given the action concomitant with the representation of $G_1$ and $G_2$ as subgroups of \GL{d}{\Z}). This is a fast procedure, and a difference in the respective cohomology groups of $G_1$ and $G_2$ shows that these groups are not conjugate in \GL{d}{\Z / |G_1|\Z} and hence not in \GL{d}{\widehat\Z}\footnote{In fact conjugacy in \GL{d}{\widehat\Z} also establishes isomorphism of the $\Z$-cohomology groups---see \cite[Lemma 4.1.9]{HP89}.}. In order to dodge potentially lengthy and unenlightening checks for isomorphisms, we actually only compare the sizes of the homology groups in the code.
		\item Decide conjugacy of $G_1$ and $G_2$ in \GL{d}{\Z[2]}.
		\item Decide conjugacy of $G_1$ and $G_2$ in \GL{d}{\Z[3]}.
		\item Decide conjugacy of $G_1$ and $G_2$ in \GL{d}{\Z[5]}.
	\end{enumerate}
	
	The number of pairs of \Z-classes whose point groups are shown not to be conjugate at each stage of the process are shown in Table \ref{ConjugacyFigures}.
	\begin{table}
		\centering
		Dimension 3:
		
		\begin{tabular}{|l|r|}
			\hline 
			Test applied  & \Z-class pairs separated \\ \hline \hline
			Size of coholomology     &         51 \\ \hline 
			Conjugacy modulo 2      &         1 \\ \hline 
			Conjugacy over \Z[2]    &         7 \\ \hline 
			Conjugacy modulo 3      &         5 \\ \hline
			Conjugacy over \Z[3]    &         0  \\ \hline 
			Conjugacy modulo 5      &         0 \\ \hline 
			Conjugacy over \Z[5]    &         0 \\ \hline \hline
			Total pairs separated   &  64 \\\hline
			Total pairs to be tested      &  64 \\ \hline
		\end{tabular}
		\vspace{1cm}
		
		Dimension 4:
		
		\begin{tabular}{|l|r|}
			\hline 
			Test applied  & \Z-class pairs separated \\ \hline \hline
			Size of coholomology     &         1105 \\ \hline 
			Conjugacy modulo 2      &         101 \\ \hline 
			Conjugacy over \Z[2]    &         137 \\ \hline 
			Conjugacy modulo 3      &         78 \\ \hline 
			Conjugacy over \Z[3]    &         0  \\ \hline 
			Conjugacy modulo 5      &         7 \\ \hline 
			Conjugacy over \Z[5]    &         0 \\ \hline \hline
			Total pairs separated   &  1428 \\\hline
			Total pairs to be tested      &  1429 \\ \hline
		\end{tabular}
		\caption{Results of tests establishing when pairs of point groups of \Z-classes (each contained in a particular \Q-class) are not conjugate over the profinite integers $\widehat{\Z}$. Note that these tests are sequential: of course all 78 cases which are not conjugate over $\Z/3\Z$ are {\it a fortiori} not conjugate over \Z[3], but have already been discarded before the test for conjugacy over \Z[3] is commenced. In dimension 2, cohomology computations alone suffice, so we have not devoted space to a table.}\label{ConjugacyFigures}
	\end{table}

	This proves sufficient to show that no two \Z-classes have point groups which are conjugate in \GL{d}{\widehat \Z}, with one exception: the two \Z-classes in the 4-dimensional \Q-class listed in the catalogue as (4, 20, 11). In this case the point groups were given an {\it ad hoc} further test and were found to not be conjugate over $\Z[2]\times \Z[3] = \Z[\{2,3\}]$.
	
	\begin{rmk} There is no contradiction here with the fact that these groups are found to be conjugate over \Z[2] and \Z[3] {\em separately}; it simply shows that the set of those $\sigma\in\Aut(G_2)$ such that $X_1$ is conjugate to $\sigma(X_2)$ over \Z[2] is disjoint from the set of those $\sigma$ allowing conjugacy over \Z[3]. 
	\end{rmk}
	\begin{rmk}
		One could ask, what if the finite groups {\em were} in fact conjugate in \GL{d}{\widehat \Z}? How could this be readily shown, without checking infinitely many primes? For a given pair of tuples $\cal A$ and $\cal B$ which are conjugate in \GL{d}{\widehat \Z}, one can establish this conjugacy in the following manner. 
		
		Compute the Smith normal form of the matrix $T_{\cal A, B}$ defined above and let $\pi$ be the set of those primes dividing some non-zero entry of the Smith normal form. Let $q=\prod_{p\in \pi}p$. The effect of this choice is that the matrix $\overline S$ from the above discussion is equal to $S$. Then the process of lifting a modulo $q$ conjugacy to a conjugacy over $\Z[\pi]$ now results in an {\em integer} matrix exhibiting this conjugacy. This integer matrix is an element of \GL{d}{\Z[\pi]}, and thus has non-zero determinant $D$. It therefore establishes conjugacy over \Z[p] for all primes not dividing $D$, leaving only finitely many primes $p$ for which conjugacy over \Z[p] remains to be checked.
	\end{rmk}
	
	\subsection{The action of ${\cal N}_{\Aut M}(G)$ on $H^2(G, \widehat M)$}\label{sec:ActionOfNorm}
	
	Let $G$ be a finite group and let $M$ be a free abelian group with a faithful $G$-action, and identify $G$ with its image in $\Aut(M)$. Let $\widehat M$ denote the profinite completion of $M$ with the induced $G$-action. Let $q = |G|$.
	
	First we note that $H^n(G, \widehat M)$ is canonically isomorphic to $H^n(G,M)$ for all $n\geq 1$. The multiplication-by-$q$ maps 
	\[H^n(G, M)\stackrel{\cdot q}{\longrightarrow} H^n(G,  M), \quad  H^n(G, \widehat M)\stackrel{\cdot q}{\longrightarrow} H^n(G, \widehat M)\]
	are trivial \cite[Proposition III.10.1]{Brown}. We then have the following commuting diagram of long exact sequences for $n\geq 1$
	\[\begin{tikzcd} 
	H^{n-1}(G, M/qM) \ar[twoheadrightarrow]{r} \ar[equal]{d} & H^n(G,M) \ar{r}{\cdot q}[swap]{0} \ar{d} & H^n(G, M) \ar[hookrightarrow]{r} \ar{d} & H^n(G, M/qM) \ar[equal]{d} \\ 
	H^{n-1}(G, M/qM) \ar[twoheadrightarrow]{r} & H^n(G, \widehat M) \ar{r}{\cdot q}[swap]{0} & H^n(G, \widehat M) \ar[hookrightarrow]{r} & H^n(G, M/qM)
	\end{tikzcd}\]
	The second vertical map shows that $H^n(G,M)\to H^n(G,\widehat M)$ is surjective, and the third shows that it is injective.

	Now let $q$ be any number such that the exponent of $H^2(G, M)$ divides $q$. Assume also that $q\geq 3$ so that $G$ necessarily embeds\footnote{This is a classical fact, attributed to Minkowski.} into $\Aut(M/qM)$. Note that such a number is $q=|G|$ (or $q=4$ if $|G|=2$). We then have a natural exact sequence
	\[\begin{tikzcd} 
	H^1(G, \widehat M) \ar{r} & H^{1}(G, M/qM) \ar[twoheadrightarrow]{r} & H^2(G, \widehat M) \ar{r} & 0
	\end{tikzcd}\]
	
	The group ${\cal N}_{\Aut(\widehat M)}(G)$ acts on each of these groups in a natural way. The action on $H^2(G,\widehat M)$ was described earlier. More generally if $N$ denotes either the module $\widehat{M}$ or one of the modules $M/qM$, then an element $\phi \in {\cal N}_{\Aut(\widehat M)}(G)$ acts on a cocycle $\zeta\in C^n(G,N)$ via 
	\[(\phi\cdot \zeta)(g_1,\ldots, g_n) = \phi(\zeta(\phi^{-1}g_1\phi, \ldots, \phi^{-1} g_n \phi)) \]
where of course $\phi\in\Aut\widehat M$ acts on $N=M/qM$ via the quotient $\Aut (\widehat M) \to \Aut (M/qM)$. These natural actions are compatible with the maps in the above exact sequence.
	
	An immediate consequence of the exact sequence above is that 
	\[{\cal N}_{\Aut(\widehat M)}(G) \cap \ker\left(\Aut(\widehat M) \to \Aut(M/qM)\right)\]
	acts trivially on $H^2(G,\widehat M)$. Thus, in computing the orbits of the action of ${\cal N}_{\Aut(\widehat M)}(G)$ on $H^2(G, \widehat M)$ we may consider the group acting to be the image of ${\cal N}_{\Aut(\widehat M)}(G)$ in $\Aut(M/qM)$.
	
	In our code, we implemented this computation in the following manner.
	\begin{enumerate}
		\item\label{Depth0} Firstly discard any \Z-class which contains only one or two space group types---this information is provided by the package \texttt{CrystCat}. Note that since the $\widehat \Z$-normalizer of $G$ certainly contains the \Z-normalizer of $G$, so the number of orbits under the $\widehat \Z$-normalizer is less than or equal to the number of space groups (i.e.\ orbits under the \Z-normalizer). If there is only one space group type, there is nothing to say. If there are two, we note that one of these orbits consists of the zero element of $H^2$ and cannot under any action by homomorphisms be carried to a non-zero element of $H^2$.
		\item Let $q= |G|$, or $q=4$ if $|G|=2$. The cohomology group $H^2(G,\widehat M) = H^2(G,M)$ is computed in the form of the quotient
		\[ H^2(G, M) = \frac{Z^1(G, M/qM)}{Z^1(G, M) + B^1(G, M/qM)} \]
		provided by the above exact sequence. The cocycle group $Z^1$ comprises the solutions, among maps $X\to M/qM$ for a generating set $X$ of $G$, to certain linear equations derived from the relators of $G$ and are readily computed.
		\item Having found $H^2(G,\widehat M)$, let $e$ be its exponent---or $e=4$ if the exponent happens to be 2. We next {\em recompute} $H^2(G,\widehat M)$ in the form 
		\[ H^2(G, M) = \frac{Z^1(G, M/eM)}{Z^1(G, M) + B^1(G, M/eM)} \]
		in which form we may readily apply the action of ${\cal N}_{\GL{d}{\Z{}/e\Z{}}}(G)$. This may seem like a redundant step (and mathematically it is), but it is a necessary concession to speed: the bottleneck in computations is the computation of the normalizer ${\cal N}_{\GL{d}{\Z{}/e\Z{}}}(G)$, the computation time of which increases exponentially\footnote{In the most extreme case we move from a normalizer computation in a group $\GL{4}{\Z/384\Z}$ of size around $3\times 10^{40}$ to a group \GL{4}{\Z/4\Z} of size around $10^9$ ---certainly a worthwhile saving!} with $e$.
		\item\label{Depth1} Compute, using the inbuilt \texttt{GAP} routines, the modulo $e$ normalizer $N={\cal N}_{\GL{d}{\Z/e\Z}}(G)$ and the number of orbits of the action of $N$ on $H^2(G,\widehat M)$. If the number of orbits under action of this group $N$ is equal to the number of space group types, we are done: making the acting group smaller cannot decrease the number of orbits.
		\item\label{Depth2} If after the previous stage we have fewer orbits than desired, we refine our computations by studying the action of a certain subgroup of $N$: the image of $N'={\cal N}_{\GL{d}{\widehat\Z}}(G)$ in $N$. This is done by the methods in the previous section: for a particular element $n\in N$, conjugacy by which takes $G$ to itself, we may decide whether it lifts to an element of \GL{d}{\widehat\Z} having the same effect. The action of the group $N'$ on $H^2(G,\widehat M)$ may now be computed, and the number of orbits is equal to the number of isomorphism types of profinite completions of space groups in this \Z-class. This may now be compared with the number of space group types.
	\end{enumerate}
	
	\begin{table}[ht]
		\centering
		Dimension 3:
		
		\begin{tabular}{|l|r|}
			\hline
			Stage of process & \Z-classes verified \\ \hline
			\ref{Depth0} &  47        \\ \hline 
			\ref{Depth1} &  25        \\ \hline 
			\ref{Depth2} &   1        \\ \hline \hline
			Total \Z-classes verified  & 73 \\  \hline
			Total number of \Z-classes & 73 \\  \hline
		\end{tabular}
		
		\vspace{1cm}
		Dimension 4:
		
		\begin{tabular}{|l|r|}
			\hline
			Stage of process & \Z-classes verified \\ \hline
			\ref{Depth0} &     340    \\ \hline
			\ref{Depth1} &     293    \\ \hline
			\ref{Depth2} &      77    \\ \hline \hline
			Total \Z-classes verified  & 710 \\ \hline
			Total number of \Z-classes & 710 \\ \hline
		\end{tabular}
		\caption{Breakdown of when it is verified that the space group types in \Z-classes are shown to have distinct profinite completions. The `stage' numbers refer to the enumerated list in Section \ref{sec:ActionOfNorm}. Dimension 2 is once again ignored for this summary: only one \Z-class survives Stage 1, and is verified at Stage 4.}\label{CohomOrbitFigures}
	\end{table}
	
The number of \Z-classes which are shown, at each stage, to have the property that the associated crystallographic group types have distinct profinite completions are tabulated in Table \ref{CohomOrbitFigures}. It is seen that in dimensions two, three and four all \Z-classes have this property. When combined with the results from the previous subsection, this completes the establishment of the following theorem.

\begin{theorem}
Let $\Gamma_1$ and $\Gamma_2$ be crystallographic groups of dimension at most 4. If $\widehat\Gamma_1\iso \widehat\Gamma_2$ then $\Gamma_1\iso\Gamma_2$.
\end{theorem}
	
\subsection*{ Acknowledgements} The authors wish to thank an anonymous referee for making many valuable suggestions. The first author was supported by a CMS/SRIM bursary from the University of Cambridge and by a bursary from King's College, Cambridge. The second author was supported by a CMS/SRIM bursary from the University of Cambridge and by a bursary from Churchill College, Cambridge. The third author was supported by a Junior Research Fellowship from Clare College, Cambridge.

\bibliographystyle{alpha}
\bibliography{bibliography}
	
\end{document}

%% file: macros.tex
\usepackage{amsmath}
\usepackage{amsfonts}
\usepackage{amssymb}
\usepackage{amsthm}
\usepackage{mathtools}
\usepackage{tikz}
\usepackage{tikz-cd}
\usepackage{enumerate}
\usetikzlibrary{patterns}
\DeclareMathOperator{\id}{id}

\DeclareMathOperator{\Aut}{Aut}

\newcommand{\iso}{\ensuremath{\cong}}
\newcommand{\Z}[1][]{\ensuremath{\mathbb{Z}_{#1}}}
\newcommand{\Q}[1][]{\ensuremath{\mathbb{Q}_{#1}}}
\newcommand{\R}{\ensuremath{\mathbb{R}}}

\newcommand{\N}{\ensuremath{\mathbb{N}}}

\newcommand{\F}{\ensuremath{\mathbb{F}}}

\newcommand{\nsgp}[1][]{\ensuremath{\triangleleft_{#1}}}

\newcommand{\gp}[1]{\ensuremath{\langle #1\rangle}}

\newcommand{\GL}[2]{\ensuremath{\mathrm{GL}_{#1}(#2)}}
\newcommand{\invlim}{\ensuremath{\varprojlim}}

%% file: environs.tex
\newtheorem{theorem}{Theorem}[section]
\newtheorem*{thmquote}{Theorem}
\newtheorem{prop}[theorem]{Proposition}

\newtheorem{clly}[theorem]{Corollary}
\newtheorem{question}[theorem]{Question}
\theoremstyle{definition}
\newtheorem{defn}[theorem]{Definition}

\theoremstyle{remark}
\newtheorem*{rmk}{Remark}

\theoremstyle{plain}

\newcounter{introthmcount}
\theoremstyle{definition}